\documentclass[onecolumn
,journal
,a4paper
]{IEEEtran}
\usepackage[T1]{fontenc}
\usepackage[cp1250]{inputenc}
\ifCLASSINFOpdf
  \usepackage[pdftex]{graphicx}
  \DeclareGraphicsExtensions{.pdf,.jpeg,.png,.jpg}
\else
\fi
\usepackage[cmex10]{amsmath}
\usepackage{amsthm}
\usepackage{amsfonts}
\usepackage{url}

\usepackage{hyperref}

  \def\one{{\mathbb I}}

    \def\BbbN{\mathbb N}
     
     \def\bN{\mathbf N}

     \def\bE{\mathbf E}
     \def\bP{\mathbf P}
     
     \def\bT{\mathbf T}
     \def\bQ{\mathbf Q}

\def\esssup{\operatornamewithlimits{ess\,sup}}

\def\ux{\underline{x}}

\def\uX{\underline{X}}

\def\cF{{\mathcal F}}

\def\cB{{\mathcal B}}
\def\E{{\mathbb E}}

\def\N{{\mathbb N}}

\def\R{{\mathbb R}}
\def\frS{{\mathfrak S}}

\newtheorem{lemma}{Lemma}[section]
\newtheorem{theorem}{Theorem}[section]
\newenvironment{classcode}{%
\small
\begin{center}
  \noindent{\bf AMS Subject Classification{\rm{:}}\ }\ignorespaces
\end{center}
}{%
\par\addvspace{26pt plus 4pt}}
\begin{document}
%
\title{A precision of the sequential change point detection${}^\star$}
\author{Aleksandra Ochman-Gozdek$^{1}$, Wojciech Sarnowski$^{2}$ and Krzysztof J. Szajowski$^{3}$
\thanks{${}^\star$The research has been supported by grant S30103/I-18. This paper was presented in part at $59^{th}$ \emph{ISI World Statistics Congress} 25--30 August 2013, Hong Kong Special Administrative Region, China in the session CPS018 (see \cite{ochsza13:random}), and at \emph{ XXXIX COnference on Mathematical Statistics}  2-6 December 2013, Wisła,Poland.}
\thanks{$^{1}$Aleksandra Ochman--Gozdek is with Toyota Motor Industries Poland Ltd. as HR  Analyst. She prepares her PhD theses at the Institute of Mathematics and Computer Science,
        Wrocław University of Technology; e-mail:
        {\tt\small Aleksandra.Ochman@pwr.wroc.pl}}%
\thanks{$^{2}$Wojciech Sarnowski is with Europejski Fundusz Leasingowy SA as Chief Specialist in Risk Management and Policy Rules Department, pl.~Orląt Lwowskich 1, 53-605 Wroc\l{}aw, Poland; e-mail:
        {\tt\small sarnowsk@interia.pl}}%
\thanks{$^{3}$Krzysztof J. Szajowski is with Faculty of Fundamental Problems of Technology, Institute of Mathematics and Computer Science,
        Wrocław University of Technology, PL-50-370 Wrocław, Poland; e-mail:
        {\tt\small Krzysztof.Szajowski@pwr.wroc.pl}}%
}
\markboth{Journal of \LaTeX\ Class Files,~Vol.~11, No.~4, December~2012}%
{A.~Ochman--Gozdek \MakeLowercase{\textit{et al.}}: A precision of the sequential change point detection	}
%



\maketitle

\begin{abstract}
A random sequence having two segments being the homogeneous Markov processes is registered. Each segment has his own transition probability law and the length of the segment is unknown and random. The transition probabilities of each process are known and \emph{a priori} distribution of the disorder moment is given. The decision maker aim is to detect the moment of the transition probabilities change. The detection of the disorder rarely is precise. The decision maker accepts some deviation in estimation of the disorder moment. In the considered model the aim is to indicate the change point with fixed, bounded error with maximal probability. The case with various precision for over and under estimation of this point is analyzed. The case when the disorder does not appears with positive probability is also included.  The results insignificantly extends range of application, explain the structure of optimal detector in various circumstances and shows new details of the solution construction. The motivation for this investigation is the modelling of the attacks in the node of networks.  The objectives is to detect one of the attack immediately or in very short time before or after it appearance with highest probability. The problem is reformulated to optimal stopping of the observed sequences. The detailed analysis of the problem is presented to show the form of optimal decision function.
\end{abstract}

\begin{classcode}
\emph{Primary: }60G40; \emph{Secondary: }{60K99; 90D60}
\end{classcode}
%
\IEEEpeerreviewmaketitle

%
%
%
%
\section{Introduction}
\IEEEPARstart{S}{uppose} that the process $X=\{X_n,n\in\BbbN\}$, $\BbbN=\{0,1,2,\ldots\}$, is observed sequentially. It is obtained from Markov processes by switching between them at a random moment $\theta$ in such a way that the process after $\theta$ starts from the state $X_{\theta-1}$. It means that the state at moment $n\in \BbbN$ has conditional distribution given the state at moment $n-1$, where the formulae describing these distributions have the different form: one for $n<\theta$ and another for $n\geq \theta$. Our objective is to detect the moment $\theta$ based on observation of $X$.  There are some papers devoted to the discrete case of such disorder detection which generalize in various directions the basic problem stated by Shiryaev~in~\cite{shi61:detection} (see e.g. Brodsky and Darkhovsky~\cite{brodar93:nonparametr}, Bojdecki~\cite{boj79:dis}, Bojdecki and Hosza~\cite{bojhos84:problem}, Yoshida~\cite{yos83:complicated}, Szajowski~\cite{sza92:detection,sza96:twodis}).

Such model of data appears in many practical problems of the quality control (see Brodsky and Darkhovsky~\cite{brodar93:nonparametr}, Shewhart~\cite{she31:quality} and in the collection of the papers \cite{basben86:abrupt}), traffic anomalies in networks (in papers Dube and Mazumdar~\cite{dubmaz01:quickest}, Tartakovsky et al.~\cite{tarroz06:intrusions}), epidemiology models (see Baron~\cite{bar04:epidemio}, Siegmund~\cite{Sie13:biology}). The aim is to recognize the moment of the change over the one probabilistic characteristics to another of the phenomenon.

Typically, the disorder problem is limited to the case of switching between sequences of independent random variables (see Bojdecki \cite{boj79:dis}). Some developments of the basic model can be found in the paper by Yakir~\cite{yak94:finite} where the optimal detection rule of the switching moment has been obtained when the finite state-space Markov chains is disordered. Moustakides~\cite{mou98:abrupt} formulates conditions which help to reduce the problem of the quickest detection for dependent sequences before and after the change to the case for independent random variables. Our result is a generalization of the results obtained by Bojdecki~\cite{boj79:dis} and Sarnowski et al.~\cite{SarSza11:transition}. It admits Markovian dependence structure for switched sequences (with possibly uncountable state-space). We obtain an optimal rule under probability maximizing criterion.

Formulation of the problem can be found in Section \ref{sformProblem}. The main result is presented in Section \ref{rozwProblem}.

\section{Formulation of the problem}\label{sformProblem}
Let $(\Omega,\mathcal{F}, \bP)$ be a probability space which supports sequence of observable random variables $\{X_n\}_{n \in \N}$ generating filtration $\mathcal{F}_n = \sigma(X_0,X_1,...,X_n)$. Random variables $X_n$ take values in $(\E, \mathcal{B})$, where $\E$ is a subset of $\R$. Space $(\Omega,\mathcal{F}, \bP)$ supports also unobservable (hence not measurable with respect to $\cF_n$) random variable $\theta$ which has the geometric  distribution:
\begin{eqnarray}
\label{rozkladyTeta}
\bP(\theta = j) = \pi\one_{\{j=0\}}+(1-\pi)p^{j-1}q\one_{\{j\geq 1\}},
\end{eqnarray}
where $q=1-p,\pi \in (0,1)$, $j=1,2,\ldots$.

We introduce in $(\Omega,\mathcal{F}, \bP)$ also two time homogeneous and independent Markov processes $\{X_n^0\}_{n \in \N}$ and $\{X_n^1\}_{n \in \N}$ taking values in $(\E, \mathcal{B})$ and assumed to be independent of $\theta$. Moreover, it is assumed that $\{X_n^0\}_{n \in \N}$ and $\{X_n^1\}_{n \in \N}$ have transition densities with respect to a $\sigma$-finite measure $\mu$, i.e., for $i = 0, 1$ and $B\in\cB$
\begin{eqnarray}\label{TransProbab}
\bP_x^{i}(X_1^{i}\in B)&=&\bP(X_1^{i}\in B|X_0^{i}=x)=\int_Bf_x^{i}(y)\mu(dy)=\int_B\mu_x(dy).
\end{eqnarray}
Random processes $\{X_n\}$, $\{X_n^0\}$, $\{X_n^1\}$ and random variable $\theta$ are connected via the rule: conditionally on $\theta = k$
\begin{equation*}
X_n=\left\{\begin{array}{ll}
X_n^0,&\text{ if $k>n$,}\\
X_{n+1-k}^1,&\text{ if $k\leq n$,}
\end{array}
\right.
\end{equation*}
where $\{X_n^1\}$ is started from $X_{k-1}^0$ (but is otherwise independent of $X^0$).

Let us introduce the following notation:
\begin{eqnarray*}
\underline{x}_{k,n} &=& (x_k, x_{k+1},\ldots,x_{n-1},x_n),\; k\leq n, \nonumber\\
L_{m}(\underline{x}_{k,n}) &=&\displaystyle{ \prod\limits_{r=k+1}^{n-m}} f_{x_{r-1}}^{0}(x_{r})\displaystyle{ \prod\limits_{r=n-m+1}^{n}}f_{x_{r-1}}^{1}(x_{r}), \nonumber\\
\underline{A}_{k,n} &=&  \times_{i=k}^{n} A_{i} = A_k \times A_{k+1} \times \ldots \times A_n,\; A_i \in \cB \nonumber
\end{eqnarray*}
where the convention $\prod_{i=j_1}^{j_2}x_i = 1$ for $j_1 > j_2$ is used.

Let us now define functions $S_\cdot(\cdot)$ and $G_\cdot(\cdot,\cdot)$
\setlength\arraycolsep{0pt}
\begin{eqnarray}\label{gestoscLaczna}
    S_n(\ux_{0,n}) &=& \pi L_n(\ux_{0,n})+\bar{\pi}\left(\sum_{i=1}^n p^{i-1}qL_{n-i+1}(\ux_{0,n}) + p^{n}L_{0}(\ux_{0,n})\right), \\
\label{funkcjaG}
    G_{l+1}(\ux_{n-l-1,n},\alpha)&=& \alpha L_{l+1}(\underline{x}_{n-l-1,n}) + (1-\alpha) \\
    && \times \left( \sum_{i=0}^{l}p^{l-i}q L_{i+1}(\ux_{n-l-1,n}) + p^{l+1}L_0(\ux_{n-l-1,n})\right). \nonumber
 \end{eqnarray}
where $x_0, x_1,\ldots, x_{n} \in \E^{n+1}, \alpha \in [0,1], 0 \leq n-l-1< n$.

The function $S(\ux_{0,n})$ stands for the joint density of the vector $\uX_{0,n}$. For any
$\underline{D}_{0,n}=\{\omega: \uX_{0,n} \in \underline{B}_{0,n}, B_i \in \cB\}$ and any $x\in \E$ we have:
\begin{eqnarray*}
\bP_x(\underline{D}_{0,n})=\bP(\underline{D}_{0,n}|X_0=x)&=&\int_{\underline{B}_{0,n}}S(\ux_{0,n})\mu(d\ux_{0,n})
\end{eqnarray*}
The meaning of the function $G_{n-k+1}(\ux_{k,n},\alpha)$ will be clear in the sequel.

Roughly speaking our model assumes that the process $\{X_n\}$ is obtained by switching at the random and unknown instant $\theta$ between two Markov processes $\{X_n^0\}$ and $\{X_n^1\}$. It means that   the first observation $X_{\theta}$ after the change depends on the previous sample $X_{\theta-1}$ through the transition pdf $f_{X_{\theta-1}}^{1}(X_{\theta})$. For any fixed $d_1,d_2 \in \{0,1,2,\ldots\}$ (the problem $\mathfrak{D}_{d_1d_2}$) we are looking for the stopping time $\tau^{*}\in \mathcal{T}$ such that
\begin{equation}
\label{PojRozregCiagowMark-Problem}
  \bP_x( -d_1\leq \theta - \tau^{*}  \leq d_2 ) = \sup_{\tau \in \frS^X} \bP_x(-d_1\leq \theta - \tau  \leq d_2 )
\end{equation}
where $\frS^X$ denotes the set of all stopping times with respect to the filtration
$\{\mathcal{F}_n\}_{n \in \N}$. Using parameters $d_i$, $i=1,2$, we control the precision level of detection. The problem $\mathfrak{D}_{dd}$, \emph{i.e.} the case $d_1=d_2=d$, when $\pi=0$ has been studied in \cite{SarSza11:transition}.

\section{\label{rozwProblem}Solution of the problem}
Let us denote:
\begin{eqnarray}
Z_n^{(d_1,d_2)} &=& \bP_x(-d_1\leq \theta - n  \leq d_2 \mid \mathcal{F}_n),\; \text{ $n=0,1,2,\ldots$}, \nonumber \\
V_n^{(d_1,d_2)} &=& \esssup_{\{\tau \in \frS^X,\;\tau \geq n\}}\bP_x( -d_1\leq \theta - \tau  \leq d_2 \mid \mathcal{F}_n), \; \text{ $n=0,1,2,\ldots$}, \nonumber \\
\label{stopIntuicyjny}
    \tau_0 &=& \inf\{ n: Z_n^{(d_1,d_2)}=V_n^{(d_1,d_2)} \}
\end{eqnarray}
Notice that, if $Z_{\infty}^{(d_1,d_2)}=0$, then $Z_{\tau}^{(d_1,d_2)} = \bP_x( -d_1\leq \theta - \tau  \leq d_2 \mid \mathcal{F}_{\tau})$ for $\tau \in \frS^X$. Since $\mathcal{F}_{n} \subseteq \mathcal{F}_{\tau}$ (when $n \leq \tau$) we have
\begin{eqnarray*}
    V_n^{(d_1,d_2)} &=& \esssup_{\tau \geq n}\bP_x(-d_1\leq \theta - \tau  \leq d_2 \mid \mathcal{F}_n)
     = \esssup_{\tau \geq n}\bE_x(\one_{\{-d_1\leq \theta - \tau  \leq d_2 \}} \mid \mathcal{F}_n) \nonumber \\
            &=& \esssup_{\tau \geq n}\bE_x(Z_{\tau}^{(d_1,d_2)} \mid \mathcal{F}_n). \nonumber
\end{eqnarray*}
The following lemma (see \cite{boj79:dis}, \cite{SarSza11:transition}) ensures existence of the solution
\begin{lemma}
    The stopping time $\tau_0$ defined by formula (\ref{stopIntuicyjny}) is the solution of problem (\ref{PojRozregCiagowMark-Problem}).
\end{lemma}
\begin{proof}
From the theorems presented in \cite{boj79:dis} it is enough to show that
$\displaystyle{\lim_{n \rightarrow \infty}Z_n^{(d_1,d_2)}=0}$. For all natural numbers $n,k$, where $n\geq k$ and $s,t\in\bar{\bN}$ we have:
\begin{eqnarray*}
Z_n^{(d_1,d_2)} &=& \bE_x(\one_{\{ -d_1\leq \theta - n  \leq d_2 \}} \mid \mathcal{F}_n) \leq \bE_x(\sup_{j \geq k}\one_{\{ -d_1\leq \theta - j \leq d_2 \}} \mid \mathcal{F}_n)
\end{eqnarray*}
From Levy's theorem
$\limsup_{n\rightarrow \infty}Z_n^{(d_1,d_2)} \leq \bE_x(\sup_{j \geq k}\one_{\{ -d_1\leq\theta - j \leq d_2 \}} \mid \mathcal{F}_{\infty})$
where $\mathcal{F}_{\infty} = \sigma\left( \bigcup_{n=1}^{\infty}\mathcal{F}_n \right)$.
It is true that: $\limsup_{j \geq k,\; k\rightarrow \infty}\one_{\{ -d_1\leq\theta - j \leq d_2 \}} = 0$ \emph{a.s.} and by the dominated convergence theorem we get
\[
  \lim_{k \rightarrow \infty}\bE_x(\sup_{j\geq k}\one_{\{ -d_1\leq\theta - j  \leq d_2 \}} \mid \mathcal{F}_{\infty} ) = 0 \text{ a.s.}.
\]
The proof of the lemma is complete.
\end{proof}
By the following lemma (see also Lemma~3.2 in \cite{SarSza11:transition}) we can limit the class of possible stopping rules to $\frS^X_{d_2+1}$ i.e. stopping times equal at least $d_2+1$. Then rule $\tilde{\tau} = \max(\tau,d_1+1)$
is at least as good as $\tau$.
\begin{lemma}\label{PojRozregCiagowMark-lematCzasStopuMax}
Let $\tau$ be a stopping rule in the problem (\ref{PojRozregCiagowMark-Problem}). Then rule $\tilde{\tau} = \max(\tau,d_1+1)$
is at least as good as $\tau$. 
\end{lemma}
\begin{proof} For $\tau \geq d_1+1$ the rules $\tau, \tilde{\tau}$ are the same. Let us consider the case when $\tau < d_1+1$. We have $\tilde{\tau} = d_1+1$ and based on the fact that $\bP_{x}(\theta \geq 1)=1$ we get:
\begin{eqnarray}
\bP_{x}(-d_1\leq \theta - \tau \leq d_2) &=& \bP_{x}(\tau - d_1  \leq \theta \leq \tau + d_2)                = \bP_{x}( 1 \leq \theta \leq \tau + d_2)                         \nonumber \\
                          &\leq& \bP_{x}( 1 \leq \theta \leq d_2+d_1+1)= \bP_{x}(\tilde{\tau} - d_1  \leq \theta \leq \tilde{\tau} + d_2)  \nonumber \\
                          &=& \bP_{x}(-d_1\leq\theta - \tilde{\tau} \leq d_2).                      \nonumber
\end{eqnarray}
This is the desired conclusion.
\end{proof}

For further considerations let us define the posterior process:
\begin{eqnarray}
    \Pi_0 &=& \pi,\nonumber\\
    \Pi_n &=& \bP_x\left(\theta \leq n \mid \mathcal{F}_n\right),\; n = 1, 2, \ldots  \nonumber
\end{eqnarray}
which is designed as information about the distribution of the disorder instant $\theta$. Next lemma transforms the payoff function to the more convenient form.
\begin{lemma}\label{ZmianaWyplaty}
    Let
\begin{eqnarray}\label{NowaWyplata}
h(\underline{x}_{0,d_1+1},\alpha) = \left( 1-p^{d_2}+ q\sum_{m=0}^{{d_1}}\frac{L_{m+1}(\underline{x}_{0,{d_1}+1}) } { p^m L_{0}(\underline{x}_{0,d_1+1})} \right) (1-\alpha),
\end{eqnarray}
where $x_0,\ldots,x_{d_1+1} \in \E$, $\alpha \in (0,1)$, then
\begin{eqnarray*}
\bP_x(-d_1\leq \theta - n  \leq d_2) = \bE_x\left[ h(\underline{X}_{n-1-d_1,n},\Pi_n) \right].
\end{eqnarray*}
\end{lemma}

\begin{proof}
We rewrite the initial criterion as the expectation
\begin{eqnarray*}
\bP_x(-d_1\leq \theta - n  \leq d_2) &=& \bE_x\left[ \bP_x(-d_1\leq \theta - n  \leq d_2 \mid \mathcal{F}_n ) \right] \nonumber\\
&=& \bE_x\left[ \bP_x( \theta \leq n+d_2 \mid \mathcal{F}_n ) - \bP_x( \theta \leq n-d_1-1 \mid \mathcal{F}_n ) \right]
\end{eqnarray*}
The probabilities under the expectation can be transformed to the convenient form using the lemmata~\ref{PojRozregCiagowMark-dodatek-lemat-pin d krokow naprzod} and \ref{PojRozregCiagowMark-dodatek-lemat-pin d-1 krokow wstecz} ( A1 and A4 of \cite{SarSza11:transition}). Next, with the help of Lemma~A5 (\emph{ibid}) (putting $l=d_1$) we can express $\bP_x( \theta \leq n+d_2 \mid \mathcal{F}_n )$ in terms of $\Pi_{n}$. Straightforward calculations imply that:
\begin{eqnarray*}
\bP_{x}( -d_1\leq\theta - n \leq d_2 \mid \cF_n ) = \left( 1-p^{d_2}+q\sum_{m=0}^{d_1} \frac{L_{m}(\underline{X}_{n-d_1-1,n})}{p^m L_{0}(\underline{X}_{n-d_1-1,n})} \right)(1-\Pi_n). 
\end{eqnarray*}
This proves the lemma.
\end{proof}

\begin{lemma}\label{FunkcjaMarkowska}
    The process $\{\eta_n\}_{n \geq d_1+1}$, where $\eta_n=(\underline{X}_{n-d_1-1,n},\Pi_n)$, forms a random Markov function.
\end{lemma}

\begin{proof}
According to Lemma 17 pp. 102--103 in \cite{shi78:optimal} it is enough to show that $\eta_{n+1}$ is a function of the previous stage $\eta_{n}$, the variable $X_{n+1}$ and that conditional distribution of $X_{n+1}$ given $\mathcal{F}_n$ is a function of $\eta_n$. Let us consider, for $x_0,\ldots,x_{d_1+2} \in \E, \alpha \in (0,1)$, a function
\begin{equation*}
\varphi(\underline{x}_{0,d_1+1},\alpha,x_{d_1+2}) = \left(\underline{x}_{1,d_1+2},\frac{f_{x_{d_1+1}}^1(x_{d_1+2})(q+p\alpha)}{G(\ux_{d_1+1,d_1+2},\alpha)}\right)
\end{equation*}
We will show that $\eta_{n+1} = \varphi(\eta_n,X_{n+1})$. Notice that we get (see  Lemma~\ref{PojRozregCiagowMark-dodatek-lemat-pin jako funkcja pin-d-1}(or Lemma~A5 in \cite{SarSza11:transition}) ($l=0$))
\begin{eqnarray}\label{PojRozregCiagowMark-JedenKrokProcesPi}
\Pi_{n+1} = \frac{f_{X_n}^1(X_{n+1})(q+p\Pi_n)}{G(\uX_{n,n+1},\Pi_n)}.
\end{eqnarray}
Hence
\begin{eqnarray*}
\varphi(\eta_n,X_{n+1}) &=& \varphi(\underline{X}_{n-d_1-1,n}, \Pi_n,X_{n+1}) 
= \left(\underline{X}_{n-d_1,n},X_{n+1}, \frac{f_{X_n}^1(X_{n+1})(q+p\Pi_n)}{G(\uX_{n,n+1},\Pi_n)}\right) \nonumber\\
&=& \left(\underline{X}_{n-d,n+1}, \Pi_{n+1}\right) = \eta_{n+1} \nonumber.
\end{eqnarray*}
Define $\hat{\cF}_n = \sigma(\theta, \uX_{0,n})$. To see that the conditional distribution of $X_{n+1}$ given $\mathcal{F}_n$ is a function of $\eta_n$, let us consider the conditional expectation of $u(X_{n+1})$ for any Borel function $u: \E \longrightarrow \R$ given $\mathcal{F}_n$. Having Lemma~A1 we get:
\begin{eqnarray*}
\bE_x(u(X_{n+1})\mid \mathcal{F}_n)
&=& \bE_x\left(u(X_{n+1})(1-\Pi_{n+1}) \mid \mathcal{F}_n\right) + \bE_x\left(u(X_{n+1})\Pi_{n+1} \mid \mathcal{F}_n\right) \nonumber \\
 &=& \bE_x\left(u(X_{n+1})\one_{\{\theta > n+1\}} \mid \mathcal{F}_n\right) + \bE_x\left(u(X_{n+1})\one_{\{\theta \leq n+1\}} \mid \mathcal{F}_n\right) \nonumber \\
 &=& \bE_x\left(\bE_x(u(X_{n+1})\one_{\{\theta > n+1\}}\mid \hat{\cF}_n )\mid \mathcal{F}_n\right) + \bE_x\left(\bE_x(u(X_{n+1})\one_{\{\theta \leq n+1\}}\mid \hat{\cF}_n ) \mid \mathcal{F}_n\right) \nonumber \\
&=& \bE_x\left(\one_{\{\theta > n+1\}}\bE_x(u(X_{n+1})\mid \hat{\cF}_n )\mid \mathcal{F}_n\right) + \bE_x\left(\one_{\{\theta \leq n+1\}}\bE_x(u(X_{n+1})\mid \hat{\cF}_n)  \mid \mathcal{F}_n\right) \nonumber \\
&=& \int_{\E} u(y)f^0_{X_n}(y)\mu(dy)\bP_x(\theta > n+1 \mid \cF_n) + \int_{\E} u(y)f^1_{X_n}(y)\mu(dy)\bP_x(\theta \leq n+1 \mid \cF_n) \nonumber\\
 &=& \int_{\E} u(y)(p(1-\Pi_n)f^0_{X_n}(y) + (q+p\Pi_n)f^1_{X_n}(y))\mu(dy) = \int_{\E} u(y)G(X_n, y, \Pi_n)\mu(dy)\nonumber
\end{eqnarray*}
This is our claim.
\end{proof}

Lemmata \ref{ZmianaWyplaty} and \ref{FunkcjaMarkowska} are crucial for the solution of the posed
problem (\ref{PojRozregCiagowMark-Problem}). They show that the initial problem can be reduced to the problem of stopping Markov random function $\eta_n = (\underline{X}_{n-d_1-1,n},\Pi_n )$ with the payoff given by the equation (\ref{NowaWyplata}). In the consequence we can use tools of the optimal stopping theory for finding the stopping time $\tau^{*}$ such that
\begin{eqnarray}\label{PrzeformulowanyProblem}
       \bE_x\left[ h(\underline{X}_{\tau^{*}-d_1-1,\tau^{*}},\Pi_{\tau^{*}}) \right] = \sup_{\tau \in \frS^X_{d+1}}\bE_x\left[ h(\underline{X}_{\tau-d_1-1,\tau},\Pi_{\tau}) \right].
\end{eqnarray}
To solve the reduced problem (\ref{PrzeformulowanyProblem}) for any Borel function
$u: \E^{d_1+2}\times [0,1]\longrightarrow \R$ let us define operators:
\small
\begin{eqnarray}
  \bT u(\underline{x}_{0,d_1+1},\alpha) &=& \bE_{x}\left[u(\underline{X}_{n-d_1,n+1},\Pi_{n+1})\mid \underline{X}_{n-1-d_1,n}=\underline{x}_{0,d_1+1},\Pi_n = \alpha \right], \nonumber\\
  \bQ u(\underline{x}_{0,d_1+1},\alpha) &=& \max\{u(\underline{x}_{0,d_1+1},\alpha), \bT u(\underline{x}_{0,d_1+1},\alpha) \} \nonumber.
\end{eqnarray}

By the definition of the operator $\bT$ and $\bQ$ we get
\begin{lemma}\label{OperatoryDlaNowejWyplaty}
    For the payoff function $h(\underline{x}_{0,d_1+1},\alpha)$ characterized by (\ref{NowaWyplata}) and for the sequence $\{r_k\}_{k=0}^{\infty}$:%
{\small
\begin{eqnarray*}
 r_0(\underline{x}_{1,d_1+1}) &=& p\left[ 1-p^{d_2} +q\sum_{m=0}^{d_1}\frac{L_{m-1}(\underline{x}_{1,d_1+1}) }{p^m L_{0}(\underline{x}_{1,d_1+1})}  \right], \nonumber\\
r_k(\underline{x}_{1,d_1+1}) &=& p\int_{\E} f_{x_{d_1+1}}^0(x_{d_1+2})\max\left\{1-p^{d_2}+q\sum_{m=1}^{d_1+1}\frac{L_{m}(\underline{x}_{1,d_1+2})}{p^m L_{0}(\underline{x}_{1,d_1+2})};r_{k-1}(\underline{x}_{2,d_1+2}) \right\}\mu(dx_{d_1+2}),\nonumber
\end{eqnarray*}
\normalsize}
the following formulae hold:%
{\small
\begin{eqnarray}
 \bQ^k h_1(\underline{x}_{1,d_1+2},\alpha) &=& (1-\alpha)\max\left\{1-p^{d_2}+q\sum_{m=1}^{d_1+1}\frac{L_{m}(\underline{x}_{1,d_1+2})}{p^m L_{0}(\underline{x}_{1,d_1+2})};r_{k-1}(\underline{x}_{2,d_1+2}) \right\},\; k \geq 1, \nonumber \\
\bT\;\bQ^k h_1(\underline{x}_{1,d_1+2},\alpha) &=& (1-\alpha)r_k(\underline{x}_{2,d_1+2}),\; k \geq 0 \nonumber.
\end{eqnarray}
\normalsize}
\end{lemma}
\begin{proof}
By the definition of the operator $\bT$ and using Lemma \ref{PojRozregCiagowMark-Pi_n_JakoFunkcja_Pi_n_minus_d_minus_1} ($l=0$) given that $(\underline{X}_{n-d_1-1,n}, \Pi_n) = (\ux_{1,d_1+2},\alpha)$ we get
{
\setlength\arraycolsep{0pt}
\begin{eqnarray*}
\bT h(\underline{x}_{1,d_1+2},\alpha) &=& \bE_{x}\left[h(\underline{X}_{n-d_1,n+1},\Pi_{n+1}) \mid \underline{X}_{n-d_1-1,n} = \ux_{1,d_1+2},\Pi_n = \alpha \right] \nonumber\\
&=& \bE_{x}\left[ \bigg(1-p^{d_2}+q\sum_{m=1}^{d_1+1}\frac{L_{m}(\underline{X}_{n-d_1,n+1})}{p^m L_{0}(\underline{X}_{n-d_1,n+1})}\bigg)(1-\Pi_{n+1}) \mid \underline{X}_{n-d_1-1,n} = \ux_{1,d_1+2},\Pi_n = \alpha \right]\nonumber\\
&=& p(1-\alpha)\int_{\E} \left(1-p^{d_2}+q\sum_{m=1}^{d_1+1}\frac{L_{m-1}(\ux_{2,d_1+2})}{p^m L_{0}(\ux_{2,d_1+2})}\frac{f_{x_{d_1+2}}^1(x_{d_1+3})}{f_{x_{d_1+2}}^0(x_{d_1+3})}\right)\\
&&\mbox{}\times\frac{f_{x_{d_1+2}}^0(x_{d_1+3})G(\ux_{d_1+2,d_1+3},\alpha)}{G(\ux_{d_1+2,d_1+3},\alpha)}\mu(dx_{d_1+3})\nonumber\\
&=&p(1-\alpha)\left[ 1-p^{d_2}+q\sum_{m=1}^{d_1+1}\int_{\E} \frac{L_{m-1}(\ux_{2,d_1+2})}{p^m L_{0}(\ux_{2,d_1+2})}f_{x_{d_1+2}}^1(x_{d_1+3})\mu(dx_{d_1+3})\right]\nonumber\\
&=&(1-\alpha)p\left[ 1-p^{d_2}+q\sum_{m=1}^{d_1+1}\frac{L_{m-1}(\ux_{2,d_1+2})}{p^m L_{0}(\ux_{2,d_1+2})}\right]=(1-\alpha)r_0(\underline{x}_{2,d_1+2}).\nonumber
\end{eqnarray*}
}
Directly from the definition of $\bQ$ we get
\begin{eqnarray*}
\bQ h(\underline{x}_{1,d_1+2},\alpha) &=& \max\left\{h(\underline{x}_{1,d_1+2},\alpha);\;\bT h(\underline{x}_{1,d_1+2},\alpha) \right\}\nonumber\\
  &=&  (1-\alpha)\max\left\{ 1-p^{d_2}+q\sum_{m=1}^{d_1+1}\frac{L_{m}(\underline{x}_{1,d_1+2})}{p^m L_{0}(\underline{x}_{1,d_1+2})} ;\;r_0(\underline{x}_{2,d_1+2})  \right\} \nonumber.
\end{eqnarray*}
Suppose now that Lemma \ref{OperatoryDlaNowejWyplaty} holds for $\bT \bQ^{k-1}h$ and $\bQ^kh$ for
some $k > 1 $. Then, using similar transformation as in the case of $k=0$, we get
{
\setlength\arraycolsep{0pt}
\begin{eqnarray}
\bT \bQ^k h(\underline{x}_{1,d_1+2},\alpha) 
 &=& \bE_{x}\left[\bQ^kh(\underline{X}_{n-d_1,n+1},\Pi_{n+1}) \mid \underline{X}_{n-d_1-1,n} = \ux_{1,d_1+2},\Pi_n = \alpha\right] \nonumber\\
 &=& \int_{\E} \! \left[\max\!\! \left\{\! 1-p^{d_2}+q \!\sum_{m=1}^{d_1+1}\frac{L_{m}(\ux_{2,d_1+3})}{p^m L_{0}(\ux_{2,d_1+3})}; r_{k-1}(\ux_{3,d_1+3})\!\right\}\!(1-\alpha)pf_{x_{d_1+2}}^0(x_{d_1+3})\!\right]\!\mu(dx_{d_1+3})\nonumber\\  &=&(1-\alpha)r_k(\ux_{2,d_1+2})\nonumber.
\end{eqnarray}
}
Moreover
\setlength\arraycolsep{0pt}
\begin{eqnarray*}
\bQ^{k+1}h(\underline{x}_{1,d_1+2},\alpha) 
  &=& \max\left\{h(\underline{x}_{1,d_1+2},\alpha);\;\bT \bQ^kh(\underline{x}_{1,d_1+2},\alpha) \right\}\nonumber\\
  &=&  (1-\alpha)\max\left\{ 1-p^{d_2}+q\sum_{m=1}^{d_1+1}\frac{L_{m}(\underline{x}_{1,d_1+2})}{p^m L_{0}(\underline{x}_{1,d_1+2})} ;\;r_k(\underline{x}_{2,d_1+2})  \right\} \nonumber.
\end{eqnarray*}
This completes the proof.
\end{proof}

The following theorem is the main result of the paper.
\begin{theorem}
      (a)  The solution of the problem (\ref{PojRozregCiagowMark-Problem}) is given by:
        \vspace{-1ex}
        \begin{eqnarray}\label{optymalnyStop}
            \tau^{*} = \inf\{n\geq d_1+1:1-p^{d_2}+q\sum_{m=1}^{d_1+1}\frac{L_{m}(\uX_{n-d_1-1,n})} { p^m L_{0}(\uX_{n-d_1-1,n})} \geq r^{*}(\underline{X}_{n-d_1,n})  \}
        \end{eqnarray}
        where $r^{*}(\underline{X}_{n-d,n}) = \lim_{k\longrightarrow \infty }r_k(\underline{X}_{n-d,n})$.

     (b)  The value of the problem, \emph{i.e.} the maximal probability for (\ref{PojRozregCiagowMark-Problem})  given $X_0 = x$, is equal to
        \setlength\arraycolsep{0pt}
        \begin{eqnarray*}
         \bP_{x}(-d_1\leq \theta - \tau^{*}  \leq d_2) 
         &=& p^{d_1+1}\int_{\E^{d_1+1}}\max\left\{1-p^{d_2}+q\sum_{m=1}^{d_1+1} \frac{L_{m}(x,\underline{x}_{1,d_1+1})}{p^m L_{0}(x,\underline{x}_{1,d_1+1})}; r^{*}(\underline{x}_{1,d_1+1}) \right\} \nonumber\\
         && \times L_0(x,\ux_{1,d_1+1})\mu(dx_1,\ldots,dx_{d_1+1}).
 \end{eqnarray*}
\end{theorem}
\begin{proof} Part (a). According to Lemma \ref{PojRozregCiagowMark-lematCzasStopuMax} we look for a stopping time equal at least $d_1+1$. From the optimal stopping theory (c.f \cite{shi78:optimal}) we know that $\tau_0$ defined by (\ref{stopIntuicyjny}) can be expressed as
\[
    \tau_0 = \inf\{n \geq d_1+1: h(\underline{X}_{n-1-d_1,n},\Pi_n) \geq \bQ^{*}h(\underline{X}_{n-1-d_1,n},\Pi_n) \}
\]
where
$\bQ^{*}h(\underline{X}_{n-1-d_1,n},\Pi_n) = \lim_{k \longrightarrow \infty} \bQ^{k}h(\underline{X}_{n-1-d_1,n},\Pi_n)$.
According to Lemma \ref{OperatoryDlaNowejWyplaty}:
\setlength\arraycolsep{0.25pt}
\begin{eqnarray*}
\tau_0 &=& \inf\left\{n \geq d_1+1: 1-p^{d_2}+q\sum_{m=1}^{d_1+1}\frac{L_{m}(\underline{X}_{n-d_1-1,n})}{p^m L_{0}(\underline{X}_{n-d_1-1,n})} \right. \nonumber \\
        && \;\;\;\;\;\;\;\;\;\;\;\left. \geq \max\{ 1-p^{d-2}+q\sum_{m=1}^{d_1+1}\frac{L_{m}(\underline{X}_{n-d_1-1,n})}{p^m L_{0}(\underline{X}_{n-d_1-1,n})} ;\;r^{*}(\underline{X}_{n-d_1,n})\} \right\} \nonumber\\
       &=& \inf \left\{n \geq d_1+1: 1-p^{d_2}+q\sum_{m=1}^{d_1+1}\frac{L_{m}(\underline{X}_{n-d_1-1,n})}{p^m L_{0}(\underline{X}_{n-d_1-1,n})} \geq r^{*}(\underline{X}_{n-d_1,n})\right\} \nonumber\\
       &=& \tau^{*}. \nonumber
\end{eqnarray*}

Part (b). Basing on the known facts from the optimal stopping theory we can write:
{ \small
        \setlength\arraycolsep{0pt}
        \begin{eqnarray*}
         \bP_{x}&(&-d_1\leq \theta - \tau^{*}  \leq d_2) 
         = \bE_{x}\left( h^{\star}_1(\underline{X}_{0,d_1+1},\Pi_{d_1+1})\right) \nonumber\\
         &=& \bE_{x}\left( (1-\Pi_{d_1+1})\max\left\{ 1-p^{d_2}+q\sum_{m=1}^{d_1+1}\frac{L_{m}(\underline{X}_{0,d_1+1})}{p^m L_{0}(\underline{X}_{0,d_1+1})} ;\;r^{\star}(\underline{X}_{1,d_1+1})  \right\} \right) \nonumber\\
         &=& \bE_{x}\left( \bE_{x}(\one_{\{\theta>d_1+1\}}\mid \cF_{d_1+1})\max\left\{ 1-p^{d_2}+q\sum_{m=1}^{d_1+1}\frac{L_{m}(\underline{X}_{0,d_1+1})}{p^m L_{0}(\underline{X}_{0,d_1+1})} ;\;r^{\star}(\underline{X}_{1,d_1+1})  \right\} \right) \nonumber\\
         &=& \bE_{x}\left( \one_{\{\theta>d_1+1\}}\max\left\{ 1-p^{d_2}+q\sum_{m=1}^{d_1+1}\frac{L_{m}(\underline{X}_{0,d_1+1})}{p^m L_{0}(\underline{X}_{0,d_1+1})} ;\;r^{\star}(\underline{X}_{1,d_1+1})  \right\}\right) \nonumber\\
         &=& \bP_{x}(\theta>d_1+1)\int_{\E^{d_1+1}}\max\left\{1-p^{d_2}+q\sum_{m=1}^{d_1+1} \frac{L_{m}(x,\underline{x}_{1,d_1+1})}{p^m L_{0}(x,\underline{x}_{1,d_1+1})}; r^{*}(\underline{x}_{1,d_1+1}) \right\} \nonumber\\
         && \times L_0(x,\ux_{1,d_1+1})\mu(dx_1,\ldots,dx_{d_1+1}) \nonumber
\end{eqnarray*}
\normalsize}
This ends the proof of the theorem.
\end{proof}

\appendix[Lemmata]

\begin{lemma}
\label{PojRozregCiagowMark-dodatek-lemat-pin d krokow naprzod}
Let $n>0$, $k \geq 0$ then:
\begin{eqnarray}
    \label{PierwszyCzlonFunkcjiWyplaty}
\bP_{x}( \theta \leq n+k \mid \cF_n ) &=& 1 -p^k(1-\Pi_n).
\end{eqnarray}
\end{lemma}
\begin{proof} It is enough to show that for $D \in \cF_n$
\begin{eqnarray}
\int_D \one_{\{\theta >n+k\}} d\bP_x = \int_D p^k(1-\Pi_n) d\bP_x. \nonumber
\end{eqnarray}
Let us define $\widetilde{\cF}_n = \sigma(\cF_n, \one_{\{\theta >n\}})$. We have:
\setlength\arraycolsep{0.1pt}
\begin{eqnarray}
 \int_D \one_{\{\theta > n+k\}}d \bP_{x}  &=& \int_D \one_{\{\theta > n+k\}}\one_{\{\theta > n\}} d \bP_x = \int_{D \cap \{\theta > n\}} \!\!\!\one_{\{\theta > n+k\}} d \bP_x \nonumber\\
        &=& \int_{D \cap \{\theta > n\}} \!\!\!\!\!\! \bE_{x}( \one_{\{\theta > n+k\}} \mid \widetilde{\cF}_n) d \bP_x = \int_{D \cap \{\theta > n\}}\!\!\!\!\!\! \bE_{x}( \one_{\{\theta > n+k\}} \mid \theta > n) d \bP_x  \nonumber\\
        &=& \int_{D} \one_{\{\theta > n\}} p^k d \bP_x = \int_D (1-\Pi_n)p^k d \bP_x \nonumber
\end{eqnarray}
\end{proof}

\begin{lemma}
For $n > 0$ the following equality holds:
\setlength\arraycolsep{0pt}
\begin{eqnarray}
\label{PIn}
\bP_{x}&(& \theta > n \mid \cF_n ) = 1 - \Pi_n = \frac{\bar{\pi}p^nL_{0}(\underline{X}_{0,n})}{S(\underline{X}_{0,n})}.
  \end{eqnarray}
\end{lemma}
\begin{proof}
Put $\underline{D}_{0,n} = \{ \omega: \uX_{0,n} \in \underline{A}_{0,n}, A_i \in \cB\}$. Then:
\begin{eqnarray}
\bP_{x}(\underline{D}_{0,n})\bP_{x}&(&\theta>n|\underline{D}_{0,n})=\int_{\underline{D}_{0,n}}\one_{\{\theta > n\}}d\bP_x = \int_{\underline{D}_{0,n}}\bP_x(\theta > n | \cF_n)d\bP_x \nonumber \\
    &=& \int_{\underline{A}_{0,n}}\frac{\bar{\pi}p^{n} L_0(\underline{x}_{0,n})}{S(\ux_{0,n})}S(\ux_{0,n})\mu(d\ux_{0,n})= \int_{\underline{D}_{0,n}}\frac{\bar{\pi}p^{n}L_0(\underline{X}_{0,n})}{S(\uX_{0,n})}d\bP_x \nonumber
\end{eqnarray}
Hence, by the definition of the conditional expectation, we get the thesis.
\end{proof}

\begin{lemma}
For $\ux_{0,l+1} \in \E^{l+2}$, $\alpha \in [0,1]$ and the functions $S$ and $G$ given by the equations (\ref{gestoscLaczna}) and
(\ref{funkcjaG}) we have:
 \begin{eqnarray}
 \label{FaktoryzacjaGestosci}
     S(\uX_{0,n}) &=& S(\uX_{0,n-l-1})G(\uX_{n-l-1,n},\Pi_{n-l-1})
 \end{eqnarray}
\end{lemma}
\begin{proof}
By (\ref{PIn}) we have
{\small
\setlength\arraycolsep{0pt}
\begin{eqnarray*}
S(\uX_{0,n-l-1})&\cdot &G(\uX_{n-l-1,n},\Pi_{n-l-1}) \nonumber\\
&=&S(\uX_{0,n-l-1}) \Pi_{n-l-1} L_{l+1}(\uX_{n-l-1,n})+ S(\uX_{0,n-l-1})(1-\Pi_{n-l-1})  \nonumber\\
&& \times
\left( \sum_{k=0}^{l}p^{l-k}q L_{k+1}(\uX_{n-l-1,n}) + p^{l+1}L_0(\uX_{n-l-1,n})\right) \nonumber\\
&\stackrel{(\ref{PIn})}{=}&
\left(S_{n-l-1}(\uX_{0,n-l-1})-\bar{\pi}p^{n-l-1}L_0(\uX_{0,n-l-1})\right)L_{l+1}(\uX_{n-l-1,n})\\
&&+\bar{\pi}p^{n-l-1}L_{0}(\uX_{0,n-l-1})\left(\sum_{k=0}^{l}p^{l-k}q L_{k+1}(\uX_{n-l-1,n})+ p^{l+1}L_0(\underline{X}_{n-l-1,n})\right)\nonumber\\
&=&\left(\pi L_{n-l-1}(\uX_{0,n-l-1})+\bar{\pi}\sum_{k=1}^{n-l-1}p^{k-1}q L_{n-k-l}(\uX_{0,n-l-1})\right)L_{l+1}(\uX_{n-l-1,n})\\
&&+ \bar{\pi}\left(\sum_{k=0}^{l}p^{n-k-1}q L_{k+1}(\uX_{0,n}) + p^{n}L_0(\uX_{0,n})\right)\nonumber\\
\end{eqnarray*}
From other side we have
\begin{eqnarray*}
S(\uX_{0,n})&=&\pi L_n(\uX_{0,n})+\bar{\pi}\left(\sum_{k=1}^{n}p^{k-1}q L_{k}(\uX_{0,n})+p^{n}L_0(\uX_{0,n})\right)\nonumber\\
&=&\pi L_n(\uX_{0,n})+\bar{\pi}\left(\sum_{k=1}^{n-l-1}p^{k-1}q L_{n-k+1}(\uX_{0,n})+ \sum_{k=n-l}^{n}p^{k-1}q L_{n-k+1}(\uX_{0,n}) + p^{n}L_0(\uX_{0,n})\right).\nonumber
\end{eqnarray*}
\normalsize}
This establishes the formula (\ref{FaktoryzacjaGestosci}).
\end{proof}

\begin{lemma}
\label{PojRozregCiagowMark-dodatek-lemat-pin d-1 krokow wstecz}
For $n >l \geq 0$ the following equation is satisfied:
\setlength\arraycolsep{0pt}
  \begin{eqnarray*}
\bP_{x}(\theta \leq n-l-1 \mid \cF_n ) &=& \frac{\Pi_{n-l-1}L_{l+1}(\underline{X}_{n-l-1,n})}{ G(\uX_{n-l-1,n},\Pi_{n-l-1})}.
  \end{eqnarray*}
\end{lemma}
\begin{proof}
Let $\underline{D}_{0,n} = \{ \omega: \uX_{0,n} \in \underline{A}_{0,n}, A_i \in \cB\}$. Then
\setlength\arraycolsep{0pt}
\begin{eqnarray*}
\bP_{x}\!\!\!&(&\!\!\!\underline{D}_{0,n})\bP_{x}(\theta>n-l-1|\underline{D}_{0,n})=\int_{\underline{D}_{0,n}}\one_{\{\theta > n-l-1\}}d\bP_x =\int_{\underline{D}_{0,n}}\bP_x(\theta > n-l-1 | \cF_n)d\bP_x \nonumber \\
    &=& \int_{\underline{A}_{0,n}}\frac{\sum_{k=n-l}^{n}\bP_x(\theta = k)L_{n-k+1}(\ux_{0,n}) + \bP_x(\theta > n)L_0(\ux_{0,n})}{S(\ux_{0,n})}S(\ux_{0,n})\mu(d\ux_{0,n})\nonumber \\
    &=& \int_{\underline{A}_{0,n}}\frac{\bar{\pi}p^{n-l-1} L_0(\ux_{0,n-l-1}) \left(\sum_{k=0}^{l}p^{l-k}qL_{k+1}(\ux_{n-l-1,n})+ p^{l+1}L_{0}(\ux_{n-l-1,n})\right)}{S(\ux_{0,n})} \nonumber\\
    && \times S(\ux_{0,n})\mu(d\ux_{0,n})\nonumber \\
    &=& \int_{\underline{D}_{0,n}}\frac{\bar{\pi}p^{n-l-1} L_0(\ux_{0,n-l-1}) \left(\sum_{k=0}^{l}p^{l-k}qL_{k+1}(\uX_{n-l-1,n})+ p^{l+1}L_{0}(\uX_{n-l-1,n})\right)}{S(\uX_{0,n})}d\bP_x \nonumber \\
    &\stackrel{(\ref{FaktoryzacjaGestosci})}{=}& \int_{\underline{D}_{0,n}}\frac{\bar{\pi}p^{n-l-1} L_0(\ux_{0,n-l-1}) \left(\sum_{k=0}^{l}p^{l-k}qL_{k+1}(\uX_{n-l-1,n})+  p^{l+1}L_{0}(\uX_{n-l-1,n})\right)}{S(\uX_{0,n-l-1})G(\uX_{n-l-1,n},\Pi_{n-l-1})}d\bP_x \nonumber \\
    &\stackrel{(\ref{PIn})}{=}& \int_{\underline{D}_{0,n}}(1-\Pi_{n-l-1})\frac{\sum_{k=0}^{l}p^{l-k}qL_{k+1}(\uX_{n-l-1,n})+ p^{l+1}L_{0}(\uX_{n-l-1,n})}{G(\uX_{n-l-1,n},\Pi_{n-l-1})}d\bP_x \nonumber
\end{eqnarray*}
This implies that:
\setlength\arraycolsep{0pt}
\begin{eqnarray}
\label{wzor2}
\bP_{x}(\theta > n-l-1 | \cF_n)  &=& (1-\Pi_{n-l-1})\\
\nonumber &&\mbox{}\times\frac{\sum_{k=0}^{l}p^{l-k}qL_{k+1}(\uX_{n-l-1,n})+ p^{l+1}L_{0}(\uX_{n-l-1,n})}{G(\uX_{n-l-1,n},\Pi_{n-l-1})}
\end{eqnarray}
Simple transformations of (\ref{wzor2}) lead to the thesis.
\end{proof}

\begin{lemma}
\label{PojRozregCiagowMark-dodatek-lemat-pin jako funkcja pin-d-1}
For $n >l \geq 0$ the recursive equation holds:
\small
\setlength\arraycolsep{0pt}
 \begin{eqnarray}
\label{PojRozregCiagowMark-Pi_n_JakoFunkcja_Pi_n_minus_d_minus_1}
\Pi_n &=&\frac{\Pi_{n-l-1} L_{l+1}(\underline{X}_{n-l-1,n}) + (1-\Pi_{n-l-1})q\sum_{k=0}^l p^{l-k}L_{k+1}(\underline{X}_{n-l-1,n})}{ G(\uX_{n-l-1,n},\Pi_{n-l-1})}
\end{eqnarray}
\end{lemma}
\begin{proof}
With the aid of (\ref{PIn}) we get:
\begin{eqnarray*}
\frac{1-\Pi_n}{1-\Pi_{n-l-1}} &=& \frac{p^{n}L_0(\underline{X}_{0,n})}{S(\uX_{0,n})}\frac{S(\uX_{0,{n-l-1}})}{p^{n-l-1}L_0(\underline{X}_{0,{n-l-1}})} = \frac{p^{l+1}L_{0}(\uX_{n-l-1,n})}{G(\uX_{n-l-1,n},\Pi_{n-l-1})} \nonumber
\end{eqnarray*}
Hence
\begin{eqnarray*}
\Pi_n &=& \frac{G(\uX_{n-l-1,n},\Pi_{n-l-1}) - p^{n-l-1}L_0(\underline{X}_{0,{n-l-1}})(1-\Pi_{n-l-1})}{G(\uX_{n-l-1,n},\Pi_{n-l-1})} \nonumber\\
    &=& \frac{\Pi_{n-l-1} L_{l+1}(\underline{X}_{n-l-1,n}) + (1-\Pi_{n-l-1})q\sum_{k=0}^l p^{l-k}L_{k+1}(\underline{X}_{n-l-1,n})}{ G(\uX_{n-l-1,n},\Pi_{n-l-1})}.\nonumber
\end{eqnarray*}
\vspace{-3ex}
This establishes the formula (\ref{PojRozregCiagowMark-Pi_n_JakoFunkcja_Pi_n_minus_d_minus_1}).
\end{proof}

\section*{Final remarks}
The presented analysis of the problem when the acceptable error for stopping before or after the disorder allows to see which protection is more difficult to control. When we admit that it is possible sequence of observation without disorder it is interesting question how to detect not only that we observe the second kind of data but that there were no at all data of the first kind. It can be verified by standard testing procedure when we stop very early ($\tau\leq \min\{d_1,d_2\}$).


\def\cprime{$'$}

\end{document}